\documentclass[11pt,oneside]{amsart}
 
\usepackage{hyperref}
\usepackage{geometry} 
\usepackage{amsmath, amsthm, amssymb}
\usepackage[dvipsnames]{xcolor}
\usepackage{graphicx}

\usepackage[normalem]{ulem}

\usepackage{graphicx} 
\usepackage{mathrsfs} 
\usepackage{tabularx}

\usepackage{color}
\usepackage{verbatim}
\usepackage[bbgreekl]{mathbbol}
\usepackage[dvipsnames]{xcolor}
\usepackage{enumitem}

\usepackage{geometry}
\makeatletter
\@addtoreset{equation}{section}
\makeatother

\usepackage[all]{xy}
\usepackage[utf8]{inputenc}
\usepackage[english]{babel}

\newcounter{cthm}

\newtheorem{thmb}[cthm]{Theorem}
\newtheorem{proposition}[equation]{Proposition}
\newtheorem{thm}[equation]{Theorem} 

\newtheorem{corr}[equation]{Corollary}
\newtheorem{lem}[equation]{Lemma}
\theoremstyle{definition}
\newtheorem{notation}[equation]{Notation}
\newtheorem{defin}[equation]{Definition}

\newtheorem{problem}[equation]{Problem}
\newtheorem{question}[equation]{Question}

\newtheorem{exam}[equation]{Example}

\newcommand{\EEE}{\mathscr{E}}
\newcommand{\OOO}{\mathscr{O}}

\let\emptyset\varnothing

\newcommand{\Addresses}{{
 \bigskip
 \footnotesize
 \textsc{Laboratory of Algebraic Geometry, Faculty of Mathematics \\ National Research University Higher School of Economics, and \\
 Independent University of Moscow}\\
 \textit{E-mail:} \texttt{kostyaloginov@gmail.com}
}}

\begin{document}
\author{Konstantin Loginov}
\title{On semistable degenerations of Fano varieties}
\thanks{Partially supported by the Russian Academic Excellence Project ’5-100’, Foundation for the Advancement of Theoretical Physics and Mathematics “BASIS”, and the Simons Foundation. }
\maketitle

\begin{abstract}Consider a family of Fano varieties $\pi: X \longrightarrow B\ni o$ over a curve germ with a smooth total space $X$. Assume that the generic fiber is smooth and the special fiber $F=\pi^{-1}(o)$ has simple normal crossings. Then $F$ is called a semistable degeneration of Fano varieties. We show that the dual complex of $F$ is a simplex of dimension $\leq \mathrm{dim}\ F$. Simplices of any admissible dimension can be realized for any dimension of the fiber. Using this result and the Minimal Model Program in dimension $3$ we reproduce the classification of the  semistable degenerations of del Pezzo surfaces obtained by Fujita. We also show that the maximal degeneration is unique and has trivial monodromy in dimension~$\leq3$. \end{abstract}

\section*{Introduction}
By a \emph{semistable family} we mean a family of projective algebraic varieties over a curve germ with a smooth total space such that the special fiber is reduced and has simple normal crossings. The semistable reduction theorem \cite{KKMS73} states that any family with a smooth generic fiber can be birationally transformed into a semistable one after a finite base change. We say that the special fiber of a semistable family is a \emph{semistable degeneration} of its generic fiber.

The dual complex (see Definition \ref{dual_complex}) of the special fiber is an important invariant of a degeneration. Its topology in some sense reflects the geometry of the generic fiber. There are many results along these lines. For example, a theorem of Kulikov \cite{Kul77} states that for a semistable degeneration of a K$3$ surfaces the dual complex can have exactly one of three types, and the maximal degeneration (such that its special fiber has the dual complex of maximal possible dimension) has a triangulation of a $2$-sphere as a dual complex. The three Kulikov's cases can be distinguished in terms of the monodromy around the special fiber. In particular, if the monodromy is trivial then every fiber of the family is smooth. We say that such family is \emph{smooth}.

It is natural to ask about the semistable degenerations of del Pezzo surfaces. In \cite{Fu90} Fujita obtained the classification of such degenerations. Later, Kachi in \cite{Ka07} used deformation theory to prove that all the cases in Fujita's list can be realized. For more results on the degenerations of surfaces see, for example, \cite{Per77}. There is also a notion of a dual complex of a singularity, see, for example, \cite{St08}.

In higher dimensions, de Fernex, Kollar and Xu showed that if the generic fiber of a semistable family is rationally connected then the dual complex of the special fiber is contractible, see \cite[Theorem 4]{dFKX12}. The main theorem of this paper is a more specific result in the case when the fibers of a semistable family are Fano varieties.

\begin{thmb}
\label{thma-A}
Let $\pi: X \longrightarrow B\ni o$ be a semistable family of $n$-dimensional Fano varieties. Then the dual complex of its special fiber $F=\pi^{-1}(o)$ is a simplex $\Delta^k$ of dimension~$k\leq~n$. In dimension $n\leq 3$ the maximal degeneration (such that $k=n$) is unique and has trivial monodromy. Moreover, it can be obtained as the blow-up of a flag of subspaces $$\{\mathrm{pt}\}=\mathbb{P}^0\subset \dots \subset \mathbb{P}^{n-1}$$ in a fiber of a smooth family whose fibers are isomorphic to $\mathbb{P}^n$. In this case, each of the $n+1$ components of the special fiber for $n=1,2,3$ is isomorphic to the blow-up of $\mathbb{P}^n$ in a flag of subspaces $$\{\mathrm{pt}\}=\mathbb{P}^0\subset \dots \subset \mathbb{P}^{n-2}.$$
\end{thmb} 
In \cite{Hu06} it is shown for any $n\geq 1$ any $k\leq n$ can be realized for some degeneration of $\mathbb{P}^n$. For $1 \leq n \leq 3$ his construction coincides with the maximal degeneration described in Theorem \ref{thma-A}. 

The special fiber of a semistable family satisfies the $d$-semistability condition introduced by Friedman in \cite{Fr83}, see Lemma~\ref{d-semistability_lemma}. We use this condition and the three-dimensional Minimal Model Program to reprove the result of \cite{Fu90}, that is, to obtain the classification of semistable degenerations of del Pezzo surfaces, see Theorem \ref{semistable_delpezzo}. This gives the case $n=2$ of the theorem (the case $n=1$ is trivial). 

The canonical line bundle is defined for simple normal crossing varieties. For such a variety $F$ we say that $F$ is Fano if $-K_F$ is ample. This is equivalent to the following condition: each component $F_j$ is log Fano (see Definition \ref{logFano}) with respect to the boundary $D_j$ given by the intersection with the other components.

In \cite{Tz15} Tziolas showed that any $d$-semistable simple normal crossing Fano variety can be smoothed, that is, included as the special fiber in a semistable family. Hence, the classification of semistable degenerations of Fano varieties is equivalent to the classification of $d$-semistable simple normal crossing Fano varieties. 

From this point of view it is important to study log Fano varieties. If we consider only smooth log Fano varieties with non-empty integral boundary then in dimension~$1$ the situation is simple: the only log Fano curve is a projective line. By contrast, already in dimension~$2$ there are infinitely many non-isomorphic log Fano varieties (they are called log del Pezzo surfaces). For example, one can take any Hirzebruch surface with the negative section as the boundary. The classification of log del Pezzo surfaces and three-dimensional log Fano varieties is contained in \cite{Ma83}. We use it to show that the maximal degeneration is unique in dimension $3$, see Proposition \ref{maximal_dim3}.

\

The paper is organised as follows. In the introduction we review some basic facts about the components of a semistable degeneration and give the necessary definitions. In section \ref{section_dual_complex} we prove that the dual complex is a simplex. Then, in Section \ref{section_delPezzo} we reprove the theorem on the classification of semistable degenerations of del Pezzo surfaces and show that the monodromy in each case is trivial. Finally, in Section \ref{section_maximal_degeneration} we prove that the maximal degeneration in dimension $3$ is unique and also has trivial monodromy.

\ 

The author would like to thank L. Soukhanov for the conversations that inspired the present paper, Yu. Prokhorov for constant support and helpful advice, I. Krylov and the Korean Institute for Advanced Study for hospitality during the work on this paper, D. Mineyev and C. Shramov for useful discussions. 

\section{Notation and conventions}
\label{section_notation}
We work over the field of complex numbers. Throughout the paper, we use the following notation:

\begin{itemize}[leftmargin=*]
\item
$\mathscr{N}_{X/Z}$: a normal bundle to a smooth subvariety $X$ in a smooth variety $Z$;
\item
$l$: a line on $\mathbb{P}^2$;
\item
$q$: a smooth conic on $\mathbb{P}^2$;
\item
$l_1, l_2$: two rulings on $\mathbb{P}^1\times\mathbb{P}^1$; 
\item
$l_{(1,1)}$: an irreducible curve of bidegree $(1,1)$ on $\mathbb{P}^1\times\mathbb{P}^1$; 
\item 
$\mathbb{F}_n$: $n\geq0$: the $n$-th Hirzebruch surface;
\item 
$s$: the $(-n)$-section on $\mathbb{F}_n$;
\item
$f$: a fiber on $\mathbb{F}_n$;
\item
$h$: a $(+n)$-section on $\mathbb{F}_n$, that is, $h\sim s + nf$;
\item
$H$: the tautological divisor on $\mathbb{P}_Z(\EEE)$ where $Z$ is a smooth variety and $\EEE$ is a vector bundle;
\item
$M$: the preimage of a hyperplane in $\mathbb{P}_{\mathbb{P}^k}(\EEE)$, $k\geq 1$, under the projection morphism;
\item
$M_f$, $M_s$: the preimage of a fiber and of the negative section, respectively, in~$\mathbb{P}_{\mathbb{F}_1}(\EEE)$ under the projection morphism.
\end{itemize}

%



\section{Preliminaries}
\label{section_prelim}
\begin{notation}\label{notation}
Let $F = \pi^{-1}(o)$ be the special fiber of a flat projective family $$\pi: X \longrightarrow B\ni o$$ over a curve germ such that the total space $X$ is smooth and the anti-canonical divisor class $-K_X$ is $\pi$-ample. Assume that $F=\sum F_i $ is reduced, reducible and has simple normal crossings (snc). Such family we call a \emph{semistable family of Fano varieties}. We keep these conventions throughout the paper. 
\end{notation}

\begin{defin}
\label{logFano}
By \emph{a log Fano variety} we mean a pair $(X, D)$ where $X$ is a smooth projective variety and $D$ is an effective $\mathbb{Q}$-divisor on $X$ with snc support whose coefficients are in $[0,1]$, called \emph{the boundary}, such that $-K_X-D$ is ample. A log Fano variety of dimension $2$ is called \emph{a log del Pezzo surface}. 
\end{defin}

Notice that the definition of a log Fano variety in the literature may be different from the one given above.

\begin{lem}
In the notation of \ref{notation} the pair
$\bigl( F_i,\ \sum_{j \neq i} F_j|_{F_i}\bigr)$ is a log Fano.
\end{lem}
\begin{proof}
Notice that $$F = \sum F_i = 0$$ over $B$. By the adjunction formula $$-K_X|_{F_i} = -K_{F_i} - \sum_{j\neq i} F_j|_{F_i}.$$ This divisor is ample by our assumption. \end{proof}

\begin{corr}
$ \bigl( F_{i_1}\cap \dots \cap F_{i_{m}},\quad \sum_{j \neq i_1, \dots, i_{m}} F_j|_{F_{i_1}\cap \dots \cap F_{i_{m}}} \bigr) $ is a log Fano variety.
\end{corr}
\begin{proof}
Analogous. The fact that $F_{i_1}\cap \dots \cap F_{i_{m}}$ is irreducible will be independently proven later (Theorem \ref{Simplex}). 
\end{proof}

\begin{corr}
\label{ratsurfaces}
Each component $F_i$ is rationally connected. In particular, if $\mathrm{dim} \ X =~3$ then $F_i$ is a rational surface.
\end{corr}
\begin{proof}
The pair $$\Bigl(F_i,\quad (1-\epsilon)\sum_{j\neq i} F_j|_{F_i} \Bigr)$$ for $0<\epsilon \ll 1$ is a klt log Fano. Then the first assertion follows from \cite{Zh06}. The second assertion follows immediately. 
\end{proof}

Similarly, any intersection~$F_{i_1}\cap \dots \cap F_{i_{m}}$ is rationally connected.

\begin{lem}
\label{d-semistability_lemma}
Let $D=F_i \cap F_j$ for $i\neq j$. Then 
\begin{equation}
\label{dss}
\mathscr{N}_{D/ F_j} \otimes \mathscr{N}_{D/ F_i} \otimes \bigotimes_{k\neq i, j} \OOO_D(F_k|_D) \simeq \OOO_D.
\end{equation}
\end{lem}
\begin{proof}
Since $\sum F_k = 0$ over $B$ one has $\sum F_k|_D = 0$. Notice that $\OOO_D(F_i|_{D})= \mathscr{N}_{D/ F_j}$ and $\OOO_D(F_j|_{S})= \mathscr{N}_{D/ F_i}$. The assertion follows.
\end{proof}

\begin{defin}[{\cite{Fr83}}]
We refer to the equation \eqref{dss} as \it{the $d$-semistability condition}. 
\end{defin}

Notice that \eqref{dss} is just a corollary of the original condition proposed by \cite{Fr83}. However, this is enough for our purposes. The particular case of 
Lemma~\ref{d-semistability_lemma} in dimension $3$ is the following

\begin{corr}[{\cite[Triple point formula]{Kul77}}]
\label{triple_point} 
Let $\mathrm{dim}\ X = 3$, and let $C = F_i \cap F_j\subset~F$. Then the following formula holds: 
\begin{equation} 
\label{triple_formula}
C|_{F_i}^2 + C|_{F_j}^2 + n_C = 0
\end{equation} 
where $n_C$ is the number of triple points of the divisor $F$ along $C$.
\end{corr}


\section{The dual complex}
\label{section_dual_complex}
\begin{defin}
\label{dual_complex}
\emph{The dual complex}, denoted by $\mathcal{D}(F)$, of a simple normal crossing divisor $F=\sum F_i$ on a smooth variety $X$ is a CW-complex whose vertices are one-to-one correspondence with the irreducible components $F_i$ of $F$ and whose $m$-faces correspond bijectively to the irreducible components of the intersection of $m + 1$ irreducible components $F_{i_1}\cap \dots \cap F_{i_{m+1}}$ for $i_1 < \dots < i_{m+1}$ (they are also called strata of~$\mathcal{D}(F)$). The attaching maps are defined in the natural way.
\end{defin}

Obviously, the dimension of $\mathcal{D}(F)$ does not exceed the dimension of $F$.

\begin{lem}[{\cite[Lemma 2.4]{Ma83}}]
\label{strongly_conn}
In the assumptions of Section \ref{section_prelim}, the special fiber $F$ is strongly connected, that is, $F_i\cap F_j\neq \emptyset$ for any~$i, j$. Moreover, $F_i\cap F_j$ is irreducible.
\end{lem}
\begin{proof}
We give a proof alternative to that of \cite{Ma83}. Consider a log Fano pair $$\Bigl(X,\quad F_i + F_j+ \sum_{k\neq i,j} (1-\epsilon) F_k\Bigr)$$ for $0 < \epsilon \ll 1$. Then $F_i\cap F_j\neq \emptyset$ by Shokurov-Kollár connectedness theorem (we use it in the simplest form, see, for example, \cite[6.50]{CKS04}). Applying the same theorem to a log Fano pair $$\Bigl(F_i,\quad F_j|_{F_i}+\sum_{k\neq i, j} (1-\epsilon) F_k|_{F_i}\Bigr)$$ we obtain that $F_i\cap F_j$ is connected. Irreducibility follows from the snc condition.
\end{proof}

\begin{thm}
\label{Simplex}
Let $\mathrm{dim}\ X = n+1$. Then the dual complex $\mathcal{D}(F)$ is a simplex $\Delta^k$ of dimension~$k\leq n$.
\end{thm}
\begin{proof}
We use induction to show that for any set of $m+1\leq n+1$ vertices $\{ i_1, \dots, i_{m+1} \}$ in $\mathcal{D}(F)$ there exists a unique $m$-simplex with these vertices. By Lemma \ref{strongly_conn} this is true for $m=1$. For any $k$ such that $1\leq k\leq m+1$ put $$Z_0 = X, \ \ Z_{k}=F_{i_1}\cap \dots \cap F_{i_{k}}.$$ Assume that the statement is true for $m$ vertices, in particular, both $Z_{m-1}\cap F_{i_m}=Z_m$ and $Z_{m-1}\cap F_{i_{m+1}}$ are non-empty. Consider a log Fano pair 
$$\Bigl(Z_{m-1},\quad F_{i_{m}}|_{Z_{m-1}}+F_{i_{m+1}}|_{Z_{m-1}}+\sum_{k\neq i_1, \dots, i_{m-1}} (1-\epsilon) F_k|_{Z_{m-1}}\Bigr)$$ and observe that by Shokurov-Kollár connectedness theorem $Z_{m+1}$ is non-empty. Next, consider a log Fano pair
$$\Bigl(Z_m,\quad F_{i_{m+1}}|_{Z_m}+\sum_{k\neq i_1, \dots, i_m} (1-\epsilon) F_k|_{Z_m}\Bigr)$$
to obtain that $Z_{m+1}$ is connected. Hence, the first assertion is proven. Restriction on the dimension follows from the snc condition.
\end{proof}

\begin{corr}
\label{dim2_case}
Let $\mathrm{dim}\ X = 2$. Then $F = F_1 + F_2$ and $F_1\simeq F_2 \simeq \mathbb{P}^1$.
\end{corr}
\begin{proof}
Follows from Theorem \ref{Simplex} and the fact that any log Fano curve is isomorphic to~$\mathbb{P}^1$.
\end{proof}

\begin{proposition}[{\cite[Example 4.3]{Hu06}}]
A simplex $\Delta^k$ of any dimension $k\leq n$ can be realised as a dual complex $\mathcal{D}(F)$ of the special fiber $F$ for a semistable family whose generic fiber is isomorphic to $\mathbb{P}^{n}$. 
\end{proposition}

This example can be obtained by blowing-up (starting from a point) a flag of subspaces $$\{\mathrm{pt}\} = \mathbb{P}^0 \subset \mathbb{P}^1 \subset \dots \subset \mathbb{P}^{n-1}$$ in the special fiber of a smooth family whose fibers are isomorphic to $\mathbb{P}^{n}$. One checks that each of $n+1$ components is isomorphic to $\mathbb{P}^n$ blown-up in a flag of subspaces $$\{\mathrm{pt}\} = \mathbb{P}^0 \subset \mathbb{P}^1 \subset \dots \subset \mathbb{P}^{n-2}.$$

\section{Degenerations of del Pezzo surfaces}
\label{section_delPezzo}
As before, we work in the assumptions~\ref{notation} of Section~\ref{section_prelim}.
\begin{lem}
\label{contr_surf}
Let $\mathrm{dim}\ X = 3$. Then for any component $F_j$ of the special fiber $F=\sum F_i$ there exists an extremal divisorial Mori contraction $f_j$ such that $\mathrm{Exc}\ f_j = F_j$. Moreover, any component of $f_j(F)$ is normal, hence $f_j$ induces a contraction on each component of $F$.
\end{lem}
\begin{proof}
Consider an ample curve $C$ in $F_j$. Then in $\overline{\mathrm{NE}}(X/B)$ one has $$C = a_1 Z_1 + \dots a_n Z_n,\ \ a_i \geq 0$$
where $Z_k$ are extremal curves that span the Mori cone which is polyhedral by the relative logarithmic cone theorem \cite[3.25]{KM98}. Since $C$ is ample and $F_j$ intersects all the other components (see Lemma~\ref{strongly_conn}) one has $$C \cdot F_j = - \sum_{i\neq j} C \cdot F_i < 0.$$ 

On the other hand, $$C\cdot F_j = (a_1 Z_1 + \dots a_n Z_n)\cdot F_j.$$ 
It follows that for some $k$ one has $Z_k \cdot F_j < 0$ and hence all the curves numerically equivalent to $Z_k$ are contained in $F_j$. Since there are no flipping contractions on a smooth threefold, the curves in the class $Z_k$ span $F_j$. Thus, the corresponding contraction (call it $f_j$) is divisorial and $\mathrm{Exc}\ f_j = F_j$. 

We prove the second assertion. Since $f_j$ is $K_X$-negative and over the base $B$, we have that $f_j$ is $(K_X+F)$-negative (recall that $F=0$ over $B$). Put $$X_j=f_j(X).$$ The pair $(X, F)$ is dlt, hence $(X_j, f_j(F))$ is dlt as well by \cite[3.44]{KM98}. By \cite[5.52]{KM98} any component of $f_j(F)$ is normal. The proof is complete. \end{proof}

Notice that the extremal Mori contractions of fiber type are also possible: take a semistable degeneration of $\mathbb{P}^1$ as in Corollary \ref{dim2_case} and multiply it by $\mathbb{P}^1$. Then the projection along $\mathbb{P}^1$ gives an extremal contraction of fiber type.

\begin{thm}[{\cite{Fu90}, \cite{Ka07}}]
\label{semistable_delpezzo}
Let $\pi: X \longrightarrow B \ni o$ be a flat proper family of surfaces over a curve germ such that the special fiber $F =\sum F_i = \pi^{-1}(o)$ is reduced, reducible and has simple normal crossings. Assume that $X$ is smooth, and $-K_X$ is $\pi$-ample. Then there are precisely $6$ possibilities for $F$, and all of them do occur. The generic fiber $X_{\eta}$ and the contractions $f_i$ given by Lemma \ref{contr_surf} are described in the following table.

\begin{center}
\newcounter{NN}
\renewcommand{\theNN}{{\rm\arabic{NN}${}^o$}}
\def\nr{\refstepcounter{NN}{\theNN}}
\begin{tabularx}{0.95\textwidth}{|l |p{0.3\textwidth} | X | l | }
\hline
 & $F_i$ & $f_i$ & $X_\eta$ \\ 
\hline
\nr\label{case:1} & $(\mathbb{P}^2, l)\ \cup (\mathbb{F}_1, s)$ & $f_1(F_1)=$smooth point & $\mathbb{P}^2$ \\ 
 & & $f_2(F_2)=\mathbb{P}^1$ & \\ 
\hline
\nr\label{case:2} & $(\mathbb{P}^2, q)\ \cup (\mathbb{F}_4, s)$ & $f_1(F_1)=$ point of type $\frac{1}{2}(1,1,1)$ & $\mathbb{P}^2$ \\ 
 & & $f_2(F_2)=\mathbb{P}^1$ & \\ 
\hline
\nr\label{case:3} & $(\mathbb{P}^1\times \mathbb{P}^1,\ l_{(1,1)})\ \cup (\mathbb{F}_2, s)$ & $f_1(F_1)=$ordinary double point & $\mathbb{P}^1\times\mathbb{P}^1$ \\ 
 & & $f_2(F_2)=\mathbb{P}^1$ & \\ 
\hline
\nr\label{case:4} & $(\mathbb{F}_1, s)\ \cup (\mathbb{F}_1, h)$ & $f_1(F_1)=\mathbb{P}^1$ & $\mathbb{F}_1$ \\ 
 & & $f_2(F_2)=\mathbb{P}^1$ & \\ 
\hline
\nr\label{case:5} & $(\mathbb{P}^1\times \mathbb{P}^1, l_1)\ \cup (\mathbb{P}^1\times \mathbb{P}^1, l_1)$ & $f_1(F_1)=\mathbb{P}^1$ & $\mathbb{P}^1\times \mathbb{P}^1$ \\ 
 & & $f_2(F_2)=\mathbb{P}^1$ & \\ 
\hline
\nr\label{case:6} & $(\mathbb{F}_1, s \cup f)\ \cup$ & $f_1(F_1)=\mathbb{P}^1$ & $\mathbb{P}^2$ \\ 
 & $(\mathbb{F}_1, s \cup f)\ \cup$ & $f_2(F_2)=\mathbb{P}^1$ & \\ 
 & $(\mathbb{F}_1, s \cup f)$ & $f_3(F_3)=\mathbb{P}^1$ & \\ 
\hline
\end{tabularx}

\

\textit{Table 1. Semistable degenerations of del Pezzo surfaces.}
\end{center}
\end{thm}

\begin{proof}
By Lemma \ref{contr_surf} for each component $F_i$ of the special fiber $F$ there exists an extremal divisorial Mori contraction $f_i: X \longrightarrow X_i$ such that $\mathrm{Exc}\ f_i = F_i$. We use the classification of such contractions \cite{Mo82} and the fact that the components are smooth and rational (see Corollary~\ref{ratsurfaces}) to obtain the following possibilities for $F_i$ and $f_i(F_i)$:
\begin{itemize}[leftmargin=*]
\item
\textbf{Type $(\mathrm{E_1})$.}
$F_i\simeq \mathbb{F}_n$, $n\geq 0$, $f_i(F_i)=\mathbb{P}^1$. 
\item
\textbf{Type $(\mathrm{E_2})$.} $F_i\simeq \mathbb{P}^2$, $f_i(F_i)=\ $smooth point, $-K_X|_{F_i} \simeq \OOO_{\mathbb{P}^2}(2)$;
\item
\textbf{Type $(\mathrm{E_3})$.} $F_i\simeq \mathbb{P}^1\times \mathbb{P}^1$, $f_i(F_i)=\ $ordinary double point, $-K_X|_{F_i} \simeq \OOO_{\mathbb{P}^1\times \mathbb{P}^1}(1,1)$;
\item
\textbf{Type $(\mathrm{E_5})$.} $F_i\simeq \mathbb{P}^2$, $f_i(F_i)=\ $singular point of type $\frac{1}{2}(1,1,1)$, $-K_X|_{F_i} \simeq \OOO_{\mathbb{P}^2}(1)$.
\end{itemize}
Now we consider these cases in detail. First, we assume that there exists a component, say $F_1\subset F$,
which is not of type $(\mathrm{E_1})$.


\subsection*{Case $(\mathrm{E_2})$.}
Then
$$F_1 \simeq \mathbb{P}^2,\qquad -K_X|_{F_1} \simeq \OOO_{\mathbb{P}^2}(2), \qquad \sum_{j\geq 2} F_j |_{F_1} \simeq \OOO_{\mathbb{P}^2}(1).$$ 
The last equality follows from the adjunction formula. In this case $f_1$ contracts~$F_1$ to a smooth point, and by Lemma \ref{strongly_conn} only one component $F_2$ intersects $F_1$ (the intersection is a curve). Thus, $F=F_1+F_2$, and there are no triple point on $F_1\cap F_2$. By the Triple point formula~\eqref{triple_formula} the restriction $F_1|_{F_2}$ is a $(-1)$-curve. By the classification given above $F_2$ is a Hirzebruch surface $\mathbb{F}_1$. This is the case~\ref{case:1} of the theorem. 

\subsection*{Case $(\mathrm{E_5})$.}
Then
$$F_1 \simeq \mathbb{P}^2, \qquad -K_X|_{F_1} \simeq \OOO_{\mathbb{P}^2}(1), \qquad \sum_{j\geq 2} F_j |_{F_1} \simeq \OOO_{\mathbb{P}^2}(2).
$$ 
Assume that $\sum_{j \geq 2} F_j |_{F_1}$ is a smooth conic. As in the previous case one has $F=F_1+F_2$. By the Triple point formula~\eqref{triple_formula} one has $F_1|_{F_2}^2=-4$. Hence $F_2$ is a Hirzebruch surface $\mathbb{F}_4$. This is the case~\ref{case:2} of the theorem. 

Now assume that $\sum_{j\geq 2}F_j |_{F_1}$ is a union of two distinct lines (a double line cannot occur since $F$ is snc on $X$). By Lemma \ref{strongly_conn} one has $F=F_1+F_2+F_3$. By the Triple point formula~\eqref{triple_formula} we have $F_1|_{F_2}^2 = F_1|_{F_3}^2= -2$. Thus, $F_2\simeq F_3 \simeq \mathbb{F}_2$, the curves~$F_2|_{F_1}$ and~$F_3|_{F_1}$ are the negative sections. Then $F_2\cap F_3$ has non-negative self-intersection on $F_2$ and $F_3$. But there is one triple point on $F_2\cap F_3$. This contradicts the Triple point formula~\eqref{triple_formula} for $F_2\cap F_3$.

\subsection*{Case $(\mathrm{E_3})$.}
Then
$$
F_1 \simeq \mathbb{P}^1\times \mathbb{P}^1,\qquad -K_X|_{F_1} \simeq \OOO_{\mathbb{P}^1\times \mathbb{P}^1}(1,1), \qquad \sum_{j\geq 2} F_j |_{F_1} \simeq \OOO_{\mathbb{P}^1\times \mathbb{P}^1}(1,1)
$$ 
and $F_1$ gets contracted to an ordinary double point. Suppose that $\sum_{j\neq 2} F_j |_{F_1}$ is irreducible. Similiar to the cases $2$ and $3$ one has $F=F_1+F_2$, $F_2\simeq \mathbb{F}_2$ and $F_1|_{F_2}$ is a negative section on $F_2$. This is the case~\ref{case:3} of the theorem.

Now assume that $\sum_{j\geq 2} F_j |_{F_1}$ is a union of two distinct lines $l_1\cup l_2$. Then, as in the case $2$, one has $F=F_1+F_2+F_3$. One checks that $F_2$ and $F_3$ are isomorphic to $\mathbb{F}_1$ and $F_2\cap F_3$ has non-negative self-intersection on $F_2$ and $F_3$. This contradicts the Triple point formula~\eqref{triple_formula} for $F_2\cap F_3$.
 
\ 
 
From now on we assume that all the contractions $f_i$ are of type $(\mathrm{E_1})$.
Hence, any component $F_i\subset F$ is a Hirzebruch surface. We show that the general ruling $h_i$ on $F_i$ (which is a fiber of $f_i$) intersects precisely one other component. Indeed, one has
\begin{equation}
\label{ruling_formula}
0=\sum F_j \cdot h_i= F_i \cdot h_i + \sum_{j\neq i} F_j \cdot h_i
\end{equation} 
and $F_i \cdot h_i = -1$ since $f_i(F_i)$ is a smooth curve. By Theorem \ref{Simplex} the number of components is either $2$ or $3$. 

\subsection*{Case: $F$ has two components. }
Thus $F=F_1+F_2$. Notice that $C=F_1 \cap F_2\simeq\mathbb{P}^1$, and by \eqref{ruling_formula} the curve $F_1 \cap F_2$ is a section on $F_1$ and $F_2$. 
Consider the contraction~$f_1$. Notice that $f_1(C)$ is not a point. Hence, $F_1$ is the blow-up of a smooth curve $f_1(F_2)\subset X_1=f_1(X)$, and $F_2$ is isomorphic to a generic fiber of the family $X_1$ (by Lemma \ref{contr_surf} we have that $f_1(F_2)$ is normal and hence $f_1(F_2)\simeq F_2$). Since the generic fiber is a del Pezzo surface one has $F_2\simeq \mathbb{F}_1$ or~$\mathbb{P}^1 \times \mathbb{P}^1$. The same argument works for the contraction $f_2$. It follows that $F_1\simeq F_2$. Using the fact that $C$ is a section on $F_1$ and $F_2$ and the Triple point formula~\eqref{triple_formula} one obtains the cases~\ref{case:4} and~\ref{case:5} of the theorem.
 
\subsection*{Case: $F$ has three components. }
Thus $F=F_1+F_2+F_3$ where each component $F_i$ is isomorphic to a Hirzebruch surface. The dual complex to $F$ is a triangle. 
By \eqref{ruling_formula} the boundary $(F_2 + F_3)|_{F_1}$ is a union of a fiber and a section on $F_1$, and similarly for $F_2$ and $F_3$. Consider the ruling $h_1$ on $F_1$ that is contained in another component, say, $F_2$. Then by the Triple point formula~\eqref{triple_formula} on $F_2$ this is a $(-1)$-curve. Hence, $F_2\simeq\mathbb{F}_1$. Similarly, $F_1\simeq F_3 \simeq \mathbb{F}_1$. Notice that $f_1(F_2)\simeq\mathbb{P}^2$, and $f_1(X)$ is the family as in the case~\ref{case:1} of the theorem. Thus we obtain the case~\ref{case:6} of the theorem.

Notice that in each case the family can be obtained as the blow-up of a smooth family $\mathrm{pr_2}: S \times \mathbb{A}^1 \longrightarrow \mathbb{A}^1$ where $S$ is a suitable del Pezzo surface. Hence all the cases occur. The proof is complete. 
\end{proof}

In all the above cases the monodromy around the special fiber is trivial since there is a sequence of contractions that leads to a smooth family of del Pezzo surfaces. 

All the components in each case are toric surfaces. However, the degeneration is toric (by which we mean that the intersections are toric curves on the components) only in the cases $1,4,5,6$. Also notice that the classification of semistable degenerations of del Pezzo surfaces can be obtained (however, without the knowledge about the contractions $f_i$) using Theorem \ref{Simplex}, the Triple point formula~\eqref{triple_formula} and the classification of log del Pezzo surfaces:

\begin{thm}[{\cite[§3]{Ma83}}]
\label{log_delPezzo}
Any log del Pezzo surface is isomorphic to one of the following: $$(\mathbb{P}^2, l),\ (\mathbb{P}^2, l\cup l'), \ (\mathbb{P}^2, q), \ (\mathbb{P}^1\times \mathbb{P}^1, l_{(1,1)}), \ (\mathbb{F}_1, h), \ (\mathbb{F}_n, s), \ n\geq 0, \ (\mathbb{F}_n, s\cup f), \ n\geq 0.$$ Conversely, any such surface is log del Pezzo.
\end{thm}

\section{An example}
By Theorem \ref{semistable_delpezzo} for a semistable family of del Pezzo surfaces $\pi: X \longrightarrow B\ni o$ there is a sequence of contractions of the components of the special fiber $F$ that leads to a smooth family. The following example shows that this is not the case in dimension $3$. Instead, we need to make some flips first.

\begin{exam}
Let $F_1 = \mathrm{Bl}_{p_1, p_2}\ \mathbb{P}^3$ be the blow-up of $\mathbb{P}^3$ in two points. Denote its exceptional divisors by $\Pi_1, \Pi_2$. Let $D_1$ be the strict transform of a plane in $\mathbb{P}^3$ containing $p_1$ and $p_2$. One has $\mathscr{N}_{D_1/F_1}=\phi^*\OOO_{\mathbb{P}^2} (1)\otimes \OOO_{D_1}(-e_1-e_2)$ where $\phi$ is the blow-down morphism from $D_1$ to a plane and $e_1, e_2$ are two $\phi$-exceptional curves. One checks that~$F_1$ is log Fano with respect to the boundary $D_1$. 

Let $F_2=\mathbb{P}_{\mathbb{P}^1\times\mathbb{P}^1}(\EEE)$ where $\EEE$ is rank $2$ vector bundle such that $c_1(\EEE)=0$ and $c_2(\EEE)=1$. Such $\EEE$ exists due to \cite[8.1.2]{Ma83}, and $F_2$ is log Fano with respect to the boundary $D_2$, the tautological divisor of $F_2$. 

One has $\mathscr{N}_{D_1/F_2}=\OOO_{D_2}(e_3)=\mathscr{N}_{D_2/F_1}^{-1}$ where $e_3$ is the ``central'' $(-1)$ curve on $D_1\simeq D_2$ which is a del Pezzo surface of degree $7$. Thus, by \cite{Tz15} the simple normal crossing Fano variety $F_1\cup F_2$ can be smoothed in a semistable family. 

One can show that any birational contraction of $X$ over the base $B$ which does not change the generic fiber $X_\eta$ does not lead to a smooth family. However, after making a flip in the union of two planes $\Pi_1, \Pi_2$ the strict transform $F'_1\simeq \mathbb{P}^3$ of the component $F_1$ can be contracted to a smooth point. 

After the contraction the strict transform $F''$ of $F_2$ is a smooth threefold of degree $(-K_{F''})^3=48$ and $\rho(F'')=2$. Hence (see \cite[Chapter 12]{IP99}), the generic fiber of the family is isomorphic to~$V_6$, that is, a divisor of bidigree $(1,1)$ on $\mathbb{P}^2\times \mathbb{P}^2$. 

Conversely, starting from a smooth family whose fibers are isomorphic to $V_6$, we may blow a point in a fiber, then make an antiflip in two copies of $\mathbb{P}^1$ and obtain the family that we started from. 
\end{exam}

It is interesting to know whether analogous construction works for any semistable family of Fano varieties, see Problem \ref{question2}.

\section{Maximal degeneration}
\label{section_maximal_degeneration}
We work in the assumption~\ref{notation} of Section \ref{section_prelim}. Consider the maximal degeneration in dimension $3$, that is, the dual complex $\mathcal{D}(F)$ is a $3$-dimensional simplex (see Theorem~\ref{Simplex}). Hence, $F=F_1+F_2+F_3+F_4$. Put $D_{ij}=F_i\cap F_j$. It turns out that such a degeneration is unique and can be explicitly described. In the next theorem we use notation of Section \ref{section_notation}. By a \emph{toric degeneration} we mean the following: all the components of the special fiber are toric varieties, and all the intersections of the components are torus-invariant. 

\begin{proposition}
\label{maximal_dim3}
In the above assumptions each component $\bigl(F_i,\ \sum_{j\neq i} F_j|_{F_i}\bigr)$ is isomorphic to 
$$\Bigl(\mathbb{P}_{\mathbb{F}_1}(\OOO_{\mathbb{F}_1}(-s-f)\oplus\OOO_{\mathbb{F}_1}),\ H\cup M_s\cup M_f\Bigr).$$ 
The boundary components are as follows: 
$$H \simeq \mathbb{F}_1, \qquad \ M_s \simeq \mathbb{P}^1\times \mathbb{P}^1, \qquad \ M_f \simeq \mathbb{F}_1.$$ 
Equivalently $F_i$ is the blow-up of $\mathbb{P}^3$ in a flag of subspaces $\{\mathrm{pt}\} =\mathbb{P}^0 \subset \mathbb{P}^1$. The degeneration is toric.
\end{proposition}
\begin{proof}
Consider a component of such degeneration, say $F_1$, and its boundary $D=D_{12}+D_{13}+D_{14}$. According to the classification \cite{Ma83}, the following cases are possible for the pair~$(F_1,D)$:
\subsection*{Case $(\mathrm{F})$}
\label{max_f}
Let $F_1$ admit an extremal contraction of type $(\mathrm{F})$, that is, $F_1$ is a Fano variety with $\rho(F_1)=1$. Then by \cite[6.1]{Ma83}
\begin{equation*}
\begin{split}
F_1\simeq \mathbb{P}^3, &\qquad D_{1i} \sim H, \qquad D_{1i} \simeq \mathbb{P}^2, \\ 
&\mathscr{N}_{D_{1i}/F_1}=\OOO_{\mathbb{P}^2}(1), \qquad 2\leq i \leq 4.
\end{split}
\end{equation*}
From the d-semistability condition if follows that 
$\mathscr{N}_{D_{1i}/F_i}=\OOO_{\mathbb{P}^2}(-3)$. It follows that the components~$F_2$, $F_3$ and $F_4$ are of type 4.1 with $a=-3$. But then the d-semistability condition fails, for example, for $D_{23}\simeq \mathbb{F}_3$. Hence, we may assume that none of the components has this type. 

\subsection*{Case $(\mathrm{E_2}$)}
\label{max_e2}
Let $F_1$ admit an extremal contraction of type $(\mathrm{E_2})$, that is, a contraction of $\mathbb{P}^2$ to a smooth point. Then by \cite[7.2]{Ma83}
\begin{equation*}
\begin{split}
F_1=\mathbb{P}_{\mathbb{P}^2}(\OOO_{\mathbb{P}^2}(-1)&\oplus\OOO_{\mathbb{P}^2}), \\
D_{12}\sim H, \qquad \ &D_{13}\sim M, \qquad \ D_{14}\sim M, \\
D_{12}&\simeq\mathbb{P}^2, \qquad \ D_{13} \simeq\mathbb{F}_1, \qquad \ D_{14} \simeq\mathbb{F}_1.
\end{split}
\end{equation*}
The variety $F_1$ and the boundary components are the same as in the case $\mathrm{(C_2)_1}$ below for $a=1$.

\subsection*{Case $(\mathrm{D_3})$}
\label{max_d3}
Let $F_1$ admit an extremal contraction of type $(\mathrm{D_3})$, that is, a $\mathbb{P}^2$-bundle over a smooth curve. Then by \cite[9.1.4]{Ma83} (notice that the contraction of type $(\mathrm{D_2})$, that is, a quadric bundle over a curve, is not possible by \cite[9.2]{Ma83})
\begin{equation*}
\begin{split}
F_1=\mathbb{P}_{\mathbb{P}^1}(\OOO_{\mathbb{P}^1}\oplus\OOO_{\mathbb{P}^1}(a_1)\oplus \OOO_{\mathbb{P}^1}(a_2))&, \qquad \ 0 \leq a_1 \leq a_2, \\
D_{12}\sim H-a_1M, \qquad \ D_{13} & \sim H-a_2M, \qquad \ D_{14} \sim M, \\
D_{12}\simeq \mathbb{F}_{a_2}&, \qquad \ D_{13}\simeq \mathbb{F}_{a_1} , \qquad \ D_{14}\simeq \mathbb{P}^2.
\end{split}
\end{equation*}
Compute the normal bundles
\begin{equation*}
\begin{split}
\mathscr{N}_{D_{12}/F_1}=\OOO_{\mathbb{F}_{a_2}}(s+(a_2-a_1)f), \ \
\mathscr{N}_{D_{13}/F_1}=\OOO_{\mathbb{F}_{a_1}}(s+(a_1-a_2)f), \ \
\mathscr{N}_{D_{14}/F_1}=\OOO_{\mathbb{P}^2}.
\end{split}
\end{equation*}
Using the $d$-semistability condition we obtain
\begin{equation}
\begin{split}
\label{dss_d3_1}
\mathscr{N}_{D_{12}/F_2}^{-1}&=\mathscr{N}_{D_{12}/F_1}\otimes \OOO_{\mathbb{F}_{a_2}}(D_{13}|_{D_{12}})\otimes \OOO_{\mathbb{F}_{a_2}}(D_{14}|_{D_{13}}) \\
&=\OOO_{\mathbb{F}_{a_2}}(s+(a_2-a_1)f) \otimes \OOO_{\mathbb{F}_{a_2}}(s) \otimes \OOO_{\mathbb{F}_{a_2}}(f) \\
&= \OOO_{\mathbb{F}_{a_2}}(2s + (a_2-a_1+1)f), \\
\end{split}
\end{equation} 
\begin{equation}
\begin{split}
\label{dss_d3_2}
\mathscr{N}_{D_{13}/F_3}^{-1}&=\mathscr{N}_{D_{13}/F_1}\otimes \OOO_{\mathbb{F}_{a_1}}(D_{12}|_{D_{13}})\otimes \OOO_{\mathbb{F}_{a_1}}(D_{14}|_{D_{13}})\\
&=\OOO_{\mathbb{F}_{a_1}}(s+(a_1-a_2)f) \otimes \OOO_{\mathbb{F}_{a_1}}(s) \otimes \OOO_{\mathbb{F}_{a_1}}(f) \\
&= \OOO_{\mathbb{F}_{a_1}}(2s + (a_1-a_2+1)f), \\
\end{split}
\end{equation} 
\begin{equation}
\begin{split}
\label{dss_d3_3}
\mathscr{N}_{D_{14}/F_4}^{-1}&=\mathscr{N}_{D_{14}/F_1}\otimes \OOO_{\mathbb{P}^2}(D_{12}|_{D_{14}})\otimes \OOO_{\mathbb{P}^2}(D_{13}|_{D_{14}}) \\
&=\OOO_{\mathbb{P}^2} \otimes \OOO_{\mathbb{P}^2}(1) \otimes \OOO_{\mathbb{P}^2}(1) \\
&=\OOO_{\mathbb{P}^2}(2).
\end{split}
\end{equation}
Since $D_{14}\simeq \mathbb{P}^2$ it follows that $F_4$ has type $\mathrm{(C_2)_1}$ with $a=-2$. Consider the component~$F_2$. Notice that $D_{12}$ and $D_{24}$ are isomorphic to a Hirzebruch surfaces $\mathbb{F}_{a_2}$ and $\mathbb{F}_2$ respectively. By \eqref{dss_d3_1} and \eqref{dss_c2_1_2} one has $\mathscr{N}_{D_{12}/F_3}=\OOO_{\mathbb{F}_{a_2}}(-2s - (a_2-a_1+1)f)$, $\mathscr{N}_{D_{24}/F_4}=\OOO_{\mathbb{F}_2}(-s-2f)$. But then $F_4$ should have type $\mathrm{(C_2)_1}$ below with $k=1$, $n=2$, $m=0$, $a_2=m$ and $(a_2-a_1+1)=0$. Hence $a_1=1$. But by assumption $a_1\leq a_2$, a contradiction.

\subsection*{Case $(\mathrm{C_2})$}
\label{max_c2}
Let $F_1$ admit an extremal contraction of type $(\mathrm{C_2})$, that is, a $\mathbb{P}^1$-bundle over a surface. By \cite[8.2]{Ma83} there are two possibilities:

\subsubsection*{Subcase $(\mathrm{C_2})_1$}

\label{max_c2_1}
The base of the contraction is $\mathbb{P}^2$. Then
\begin{equation*}
\begin{split}
F_1 = \mathbb{P}_{\mathbb{P}^2}(\OOO_{\mathbb{P}^2}(&-a)\oplus\OOO_{\mathbb{P}^2}), \qquad \ a\geq 0, \\ 
D_{12} \sim H&, \qquad \ D_{13} \sim M, \qquad \ D_{14} \sim M, \\
D_{12} &\simeq \mathbb{P}^2, \qquad \ D_{13} \simeq\mathbb{F}_a, \qquad \ D_{14} \simeq \mathbb{F}_a.
\end{split}
\end{equation*}
Compute the normal bundles
\begin{equation*}
\begin{split}
\mathscr{N}_{D_{12}/F_1} = \OOO_{\mathbb{P}^2}(-a), \ \qquad 
\mathscr{N}_{D_{13}/F_1} = \OOO_{\mathbb{F}_a}(f), \ \qquad 
\mathscr{N}_{D_{14}/F_1} = \OOO_{\mathbb{F}_a}(f).
\end{split}
\end{equation*}
Using the $d$-semistability condition we obtain
\begin{equation}
\begin{split}
\label{dss_c2_1_1}
\mathscr{N}_{D_{12}/F_2}^{-1}&=\mathscr{N}_{D_{12}/F_1}\otimes \OOO_{\mathbb{P}^2}(D_{13}|_{D_{12}}) \otimes \OOO_{\mathbb{P}^2} (D_{14}|_{D_{12}}) \\
&= \OOO_{\mathbb{P}^2}(-a) \otimes \OOO_{\mathbb{P}^2}(1) \otimes \OOO_{\mathbb{P}^2} (1) \\
&= \OOO_{\mathbb{P}^2} (2-a), \\
\end{split}
\end{equation} 
\begin{equation}
\begin{split}
\label{dss_c2_1_2}
\mathscr{N}_{D_{13}/F_3}^{-1}&=\mathscr{N}_{D_{13}/F_1}\otimes \OOO_{\mathbb{F}_a}(D_{12}|_{D_{13}}) \otimes \OOO_{\mathbb{F}_a}(D_{14}|_{D_{13}}) \\
&=\OOO_{\mathbb{F}_a}(f) \otimes \OOO_{\mathbb{F}_a}(s) \otimes \OOO_{\mathbb{F}_a}(f) \\
&= \OOO_{\mathbb{F}_a}(s + 2f), \\
\end{split}
\end{equation} 
\begin{equation}
\begin{split}
\label{dss_c2_1_3}
\mathscr{N}_{D_{14}/F_4}^{-1}&=\mathscr{N}_{D_{14}/F_1}\otimes \OOO_{\mathbb{F}_a}(D_{12}|_{D_{14}}) \otimes \OOO_{\mathbb{F}_a}(D_{13}|_{D_{14}}) \\
&=\OOO_{\mathbb{F}_a}(f) \otimes \OOO_{\mathbb{F}_a}(s) \otimes \OOO_{\mathbb{F}_a}(f) \\
&= \OOO_{\mathbb{F}_a}(s + 2f).
\end{split}
\end{equation} 
One checks that the component $F_2$ either has type $\mathrm{(D_3)}$ or $\mathrm{(C_2)_1}$. The first possibility was already excluded above. Assume that $F_2$ is of type $\mathrm{(C_2)_1}$. Consider the component~$F_3$. Notice that $D_{13}$ and $D_{23}$ are isomorphic to a Hirzebruch surfaces $\mathbb{F}_a$ and $\mathbb{F}_{a'}$ respectively. By \eqref{dss_c2_1_2} one has $\mathscr{N}_{D_{13}/F_3}=\OOO_{\mathbb{F}_a}(-s-2f)$ and $\mathscr{N}_{D_{23}/F_3}=\OOO_{\mathbb{F}_{a'}}(-s-2f)$. But there are no components with such normal bundles in this list, a contradiction.

\subsubsection*{Subcase $(\mathrm{C_2})_2$}

\label{max_c2_2}
The base of the contraction is a Hirzebruch surface, see \cite[8.4]{Ma83}. Notice that the computations in \cite[8.3, 8.4]{Ma83} are not correct. We have 
\begin{equation*}
\begin{split}
F_1 = \mathbb{P}_{\mathbb{F}_n} (\OOO_{\mathbb{F}_n}( -k&s - (kn+m)f )\oplus \OOO_{\mathbb{F}_n}) \\
D_{12} \sim H&,\ \ D_{13} \sim M_f, \ D_{14} \sim M_s. \\
& D_{12} \simeq \mathbb{F}_n, \ D_{13} \simeq \mathbb{F}_k, \ D_{14} \simeq \mathbb{F}_m, \\
\end{split}
\end{equation*}
where $n\geq 0$, $k\geq 0$, $m\geq0$.
Compute the normal bundles:
\begin{equation}
\label{normal_bundles_c2_2}
\mathscr{N}_{D_{12}/F_1}=\OOO_{\mathbb{F}_n}(-ks - (kn+m)f), \ \ \mathscr{N}_{D_{13}/F_1}=\OOO_{\mathbb{F}_k}, \ \ \mathscr{N}_{D_{14}/F_1}=\OOO_{\mathbb{F}_m}(-nf).
\end{equation}
Since above we have excluded all the other cases, we may assume that all the components $F_i$ are of type $\mathrm{(C_2)_2}$. Using the $d$-semistability condition we obtain
\begin{equation}
\begin{split}
\label{dss_c2_2_1}
\mathscr{N}_{D_{12}/F_2}^{-1} &= \mathscr{N}_{D_{12}/F_1}\otimes \OOO_{\mathbb{F}_n}(D_{13}|_{D_{12}}) \otimes \OOO_{\mathbb{F}_n}(D_{14}|_{D_{12}}) \\ 
&= \OOO_{\mathbb{F}_n}(-ks - (kn+m)f) \otimes \OOO_{\mathbb{F}_n}(s) \otimes \OOO_{\mathbb{F}_n}(f) \\
&= \OOO_{\mathbb{F}_n}((1-k)s + (1-kn-m)f), \\
\end{split}
\end{equation}
\begin{equation}
\begin{split}
\label{dss_c2_2_2}
\mathscr{N}_{D_{13}/F_3}^{-1} &= \mathscr{N}_{D_{13}/F_1}\otimes \OOO_{\mathbb{F}_k}(D_{12}|_{D_{13}}) \otimes \OOO_{\mathbb{F}_k}(D_{14}|_{D_{12}}) \\ 
&= \OOO_{\mathbb{F}_k} \otimes \OOO_{\mathbb{F}_k}(s) \otimes \OOO_{\mathbb{F}_k}(f) \\
&= \OOO_{\mathbb{F}_k}(s+f), \\
\end{split} 
\end{equation}
\begin{equation}
\begin{split}
\label{dss_c2_2_3}
\mathscr{N}_{D_{14}/F_4}^{-1} &= \mathscr{N}_{D_{14}/F_1}\otimes \OOO_{\mathbb{F}_m}(D_{12}|_{D_{14}}) \otimes \OOO_{\mathbb{F}_m}(D_{13}|_{D_{14}}) \\
&= \OOO_{\mathbb{F}_m}(-nf) \otimes \OOO_{\mathbb{F}_m}(s) \otimes \OOO_{\mathbb{F}_m}(f) \\
&= \OOO_{\mathbb{F}_m}(s + (1-n)f).
\end{split}
\end{equation}
Suppose that $n=0$. Then by \eqref{dss_c2_2_2} and \eqref{dss_c2_2_3} one has that $\mathscr{N}_{D_{13}/F_3}$ and $\mathscr{N}_{D_{14}/F_4}$ are of the form $\OOO(-s-f)$. Using the above formulas one checks that for $F_1$ one has $n=0, k=1, m=1$, for $F_2$ and $F_2$ one has $n=1, k=1, m=0$. Then the d-semistability condition fails for $D_{24}\simeq D_{34} \simeq \mathbb{P}^1\times \mathbb{P}^1$.

Thus we may assume that we have $n>0$ for each component $F_i$. Then from the above formulas it follows that one of $D_{1i}$ has to have the normal bundle of the form $\OOO(-s-f)$. It is possible only for $D_{12}$. Thus, $k=1$, $kn+m=n+m=1$. Since $n>0$ one has $n=1$, $m=0$. This argument works for any component. The claim follows.
\end{proof}


\begin{proof}[Proof of Theorem \ref{thma-A}]
The first claim follows from Theorem \ref{Simplex}. Uniqueness of the degeneration in dimensions $1, 2$ and $3$ follows from Corollary \ref{dim2_case}, Theorem \ref{semistable_delpezzo} and Proposition \ref{maximal_dim3}. Notice that these degenerations coincide with the degeneration described in \cite[Example 3.6]{Hu06} which can be obtained by blowing up a flag of subspaces in a family with smooth fibers isomorphic to a projective space. Hence, in these cases the monodromy is trivial. 
\end{proof}

We conclude by formulating the following questions and problems.

\begin{question}
\label{question1}
Is the maximal semistable degeneration of Fano varieties unique in any dimension?
\end{question}

\begin{question}
\label{question2}
Is it true that any semistable family of Fano varieties after applying the MMP (for a suitable sequence of extremal rays) becomes a smooth family? This would imply that the monodromy is trivial in all cases.
\end{question}

\begin{problem}
\label{question3}
Classify the semistable degenerations of Fano varieties in dimension~$3$.
\end{problem}


\def\cprime{$'$} \def\mathbb#1{\mathbf#1}

\Addresses
\end{document}